\documentclass[11pt]{amsart}

\usepackage{graphicx}
\usepackage{amsfonts}
\usepackage{epsf}
\usepackage{amssymb}
\usepackage{amsmath}
\usepackage{amscd}
\usepackage{hyperref}
\usepackage{color}
\usepackage{bbm}
\usepackage{pinlabel}
\usepackage{subcaption}
\usepackage{tikz}

\newtheorem{theorem}{Theorem}[section]
\newtheorem{proposition}[theorem]{Proposition}
\newtheorem{lemma}[theorem]{Lemma}

\newtheorem{corollary}[theorem]{Corollary}

\theoremstyle{remark}
\newtheorem{definition}[theorem]{Definition}
\newtheorem{example}[theorem]{Example}

\begin{document}

\title[]{Comparing nonorientable three genus and \\nonorientable four genus of torus knots}

\author{Stanislav Jabuka}
\author{Cornelia A. Van Cott}

\maketitle

\begin{abstract}
We compare the values of the nonorientable three genus (or, crosscap number) and the nonorientable four genus of torus knots. In particular, let $T(p,q)$ be any torus knot with $p$ even and $q$ odd. The difference between these two invariants on $T(p,q)$ is at least $\frac{k}{2}$, where $p = qk + a $ and $0< a < q$ and $k\geq0$. Hence, the difference between the two invariants on torus knots $T(p,q)$ grows arbitrarily large for any fixed odd $q$, as $p$ ranges over values of a fixed congruence class modulo $q$. This contrasts with the orientable setting. Seifert proved that the orientable three genus of the torus knot $T(p,q)$ is $\frac{1}{2}(p-1)(q-1)$, and Kronheimer and Mrowka later proved that the orientable four genus of $T(p,q)$ is also this same value. 
\end{abstract}



\section{Introduction}
The {\em nonorientable three genus} (or, {\em crosscap number}) $\gamma_3(K)$ of a knot $K$ in $S^3$ is the smallest first Betti number of all nonorientable surfaces $\Sigma$ embedded in $S^3$ and with $\partial \Sigma =K$. This invariant was first defined and studied by Clark \cite{Clark:1978} in 1978. Similarly, for any knot $K$, the {\em nonorientable smooth four genus $\gamma_4(K)$} is defined as the smallest first Betti number of any nonorientable surface $F$ smoothly and properly embedded in the 4-ball and with $\partial F = K$. This knot invariant $\gamma_4(K)$ was introduced by Murakami and Yasuhara \cite{MurakamiYasuhara:2000} in the year 2000. Since the interior of any surface in $S^3$ can always be pushed into $B^4$, we see that 
$$1\leq \gamma_4(K) \le \gamma_3(K). $$

In this paper, we will discuss and compare the value of these two invariants $\gamma_3$ and $\gamma_4$ on torus knots $T(p,q)$. Unless otherwise noted, throughout our discussion we use the convention that if $pq$ is even, then we take $p$ even and $q$ odd. If $pq$ is odd, then we take $p>q$. 

In the analogous situation where one looks for {\em orientable} surfaces of minimal genus, Seifert~\cite{Seifert} proved that the orientable three genus of the torus knot $T(p,q)$ is $\frac{1}{2}(p-1)(q-1)$. Subsequently, Kronheimer and Mrowka~\cite{KronheimerMrowka1:1993} proved that the orientable four genus of $T(p,q)$ is this same value. Finding an orientable surface that realizes these invariants for $T(p,q)$  is not difficult. In particular, applying Seifert's algorithm to the standard torus knot diagram will produce a genus minimizing surface. In comparison, the situation with nonorientable surfaces  is less straightforward. The two invariants $\gamma_3$ and $\gamma_4$  coincide on some torus knots, but differ on others. And, nonorientable surfaces that realize the invariants' values are less easily procured. 

Batson~\cite{Batson:2014} studied $\gamma_3$ and $\gamma_4$ on torus knots and proved that:
$$\gamma_3(T(2k, 2k-1)) = k \hspace{.5cm} \textrm{ and } \hspace{.5cm} \gamma_4(T(2k, 2k-1)) = k-1.$$ 
It follows that $\gamma_4$ can be arbitrarily large and also that $\gamma_3$ and $\gamma_4$ need not be equal. 

In this paper, we prove the following.

\begin{theorem}\label{theorem}
Let $T(p,q)$ be any nontrivial torus knot with $p$ even. Write $p = qk + a $, where $0< a < q$ and $k\geq0$.  Then we have
$$(\gamma_3 - \gamma_4)(T(p,q)) \geq \frac{k}{2}.$$
\end{theorem}

From this theorem, we have the following immediate corollary. 
\begin{corollary}\label{corollary}
The difference between $\gamma_3(T(p,q))$ and $\gamma_4(T(p,q))$ grows arbitrarily large for any fixed odd $q$ where $p$ ranges over all integer values of a fixed congruence class modulo $q$. 
\end{corollary}

The first case of the corollary (that is, when $q = 3$) can be observed easily by piecing together known results. In this paper, we do a careful study of the surfaces and invariants involved which enables us to show that the result hold generally for all odd $q$. 

This paper is organized as follows. In Section~\ref{band-move}, we describe and study a particular nonorientable band move on torus knots. This band move gives rise to three different nonorientable surface constructions for torus knots. The first surface construction is given in Section~\ref{surface-construction}. In this case, the resulting surface lives in $B^4$, and hence is a candidate for realizing $\gamma_4$. In Section~\ref{3D-constructions}, we describe two surface constructions which realize the nonorientable three genus $\gamma_3$ for torus knots. The similarity of the three constructions given here facilitates comparisons of $\gamma_3$ and $\gamma_4$, which we do in Section~\ref{compare3and4}.


\section{A nonorientable band move on torus knots}~\label{band-move}

Torus knots have the convenient property that they admit a nonorientable band move that results in another torus knot. This operation is first explicitly described by Batson in~\cite{Batson:2014}. We review and study the operation in this section. 

Draw a torus knot on the flat torus as in Fig.~\ref{FigurePinchingMove}(a). Insert a band between any two adjacent strands (Fig.~\ref{FigurePinchingMove}(b)). Now the knot which results from doing the associated band move is again a curve embedded on the torus (Fig.~\ref{FigurePinchingMove}(c)); hence it is again a torus knot. We call this particular band move a {\em pinch move}, since it has the effect of pinching two adjacent strands together. Observe that this is a nonorientable band move and also that this is a well-defined move, meaning that the result of this process does not depend on the pair of adjacent strands chosen.

\begin{figure}   
\centerline{\includegraphics[width=11cm]{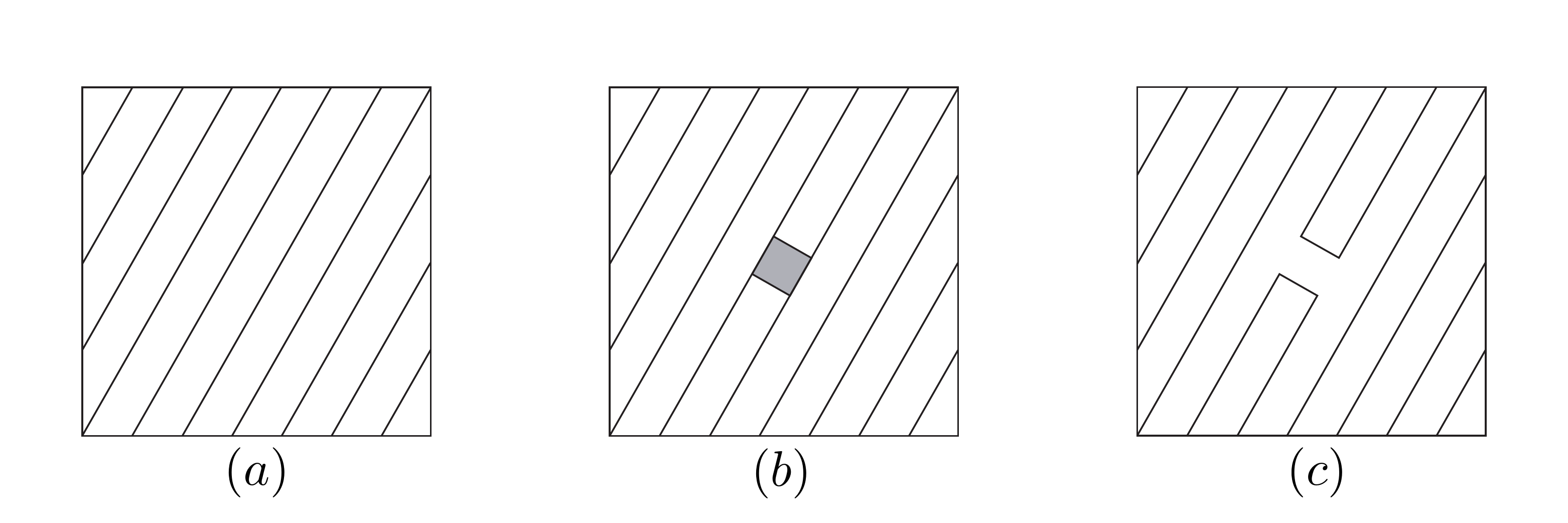}}
\vspace*{8pt}
\caption{(a) The torus knot $T(7,4)$. (b) A band between adjacent strands of the diagram. (c) The knot resulting from the band move is the torus knot $T(3,2)$. We call  a band move such as this a {\em pinch move}.  }
\label{FigurePinchingMove}
\end{figure}

Batson \cite{Batson:2014} stated that the torus knot obtained from doing a pinch move on $T(p,q)$ is given as follows (see~\cite{JabukaVC:2018} for a proof).

\begin{lemma}\cite{Batson:2014} \label{LemmaOnTheProofOfBatsonsPinchingMoveFormula}
Let $p,q>0$ be relatively prime. Begin with a diagram of the torus knot $T(p, q)$ on the flat torus. Apply a pinch move to $T(p,q)$. The resulting torus knot (up to orientation) is $T(|p-2t|, |q-2h|)$ where $t$ and $h$ are the integers uniquely determined by the requirements 
\begin{align*}
t & \equiv -q^{-1}\,(\text{mod }p)\qquad\text{ and } \qquad   t\in \{0,\dots, p-1\},\cr
h & \equiv \phantom{-} p^{-1}\,(\text{mod }q) \qquad \text{ and } \qquad h\in \{0,\dots, q-1\}.
\end{align*}
\end{lemma}

The formula for the result of a pinch move given in Lemma~\ref{LemmaOnTheProofOfBatsonsPinchingMoveFormula} can be recast using continued fractions. We first recall the fundamentals of continued fractions. A full overview can be found in~\cite{Khinchin:1997}. For $p,q>0$ relatively prime integers, consider the continued fraction expansion of $\frac{p}{q}$:
$$\frac{p}{q} = [c_0,c_1,\dots, c_m]: = c_0+  \cfrac{1}{c_1+\cfrac{1}{c_2+\cfrac{1}{\ddots +\cfrac{1}{c_{m-1} +\cfrac{1}{c_m}}}}}.$$
We shall say that the continued fraction expansion is in {\em canonical form} if $c_0\ge 0$, $c_i\ge 1$ for $i=1,\dots, m-1$, and $c_m\ge 2$. Moreover, we say that the canonical continued fraction expansion for $1$ is $[1]$, and the canonical continued fraction expansion of $0$ is $[0]$. Under these conditions, the coefficients $c_0,\dots, c_m$ are uniquely determined by $p$ and $q$. 

Suppose we truncate this continued fraction expansion as follows: $[c_0,\dots, c_i]$, where $0\leq i \leq m$. In this way, one obtains the so-called {\em $i^{th}$ convergent of $[c_0, \dots, c_m]$}. This new continued fraction expansion corresponds to a rational number $\frac{p_i}{q_i}$, where $p_i$ and $q_i$ are uniquely determined, if we require them to be positive and relatively prime. In the special case that $c_0=0$, we note $\frac{p_0}{q_0}=[0] = 0.$ In this case, we take $p_0=0$ and $q_0 = 1$. These integers $p_i, q_i$ satisfy the following recursive relation for all $2\leq i\leq m$ (see \cite{Khinchin:1997} for more discussion):
\begin{equation} \label{EquationRecursiveDefinitionOfTheithConvergent}
p_i=c_ip_{i-1}+p_{i-2}\quad  \text{ and  } \quad q_i=c_iq_{i-1}+q_{i-2}.
\end{equation}

Now we define an operation on a continued fraction expansion, termed a {\em step}. This operation was first discussed in Section 10 of \cite{BredonWood:1969}. Although originally the operation was only defined for fractions $\frac{p}{q}$ with $p$ even, the operation can be defined more generally as follows.\\

\begin{definition}~\label{StepDefinition}
Let $p$ and $q$ be relatively prime positive integers with $q \neq 2$. Let $\frac{p}{q}$ have canonical continued fraction expansion $[c_0, \dots, c_m]$. The transition of this continued fraction to the fraction $[c_0,\dots, c_m-2]$ (which must be then put in canonical form) is a {\em step}.
\end{definition}

If the continued fraction expansion resulting from a {\em step} is not in canonical form (which occurs when $c_m=2$ or $c_m=3$), we modify it so as to be in canonical form by using one of these identities (more than once, if necessary):
\begin{align} \label{EquationMakingTheContinuedFractionExpansionCanonical}
[c_0,\dots, c_{m-2},c_{m-1},0] & = [ c_0,\dots, c_{m-2}], \cr  
[c_0,\dots, c_{m-2},c_{m-1},1] & = [ c_0,\dots, c_{m-1}+1].
\end{align}

Now we are ready to give an alternative to Lemma~\ref{LemmaOnTheProofOfBatsonsPinchingMoveFormula} for the torus knot that results from a pinch move.

\begin{proposition}~\label{pinch-step}
Let $p,q$ be relatively prime positive integers with $q\neq 2$, and let $T(r,s)$ be obtained from $T(p,q)$ via a pinch move. If the rational number $\frac{p}{q}$ has continued fraction expansion $\frac{p}{q}=[c_0,\dots, c_m]$, then a continued fraction expansion of $\frac{r}{s}$ is given by $\frac{r}{s}=[c_0,\dots, c_m-2]$, which uniquely determines the values of $r$ and $s$. Hence, we have the commutative diagram, in which the vertical maps associate to the torus knot $T(x,y)$ the rational number $\frac{x}{y}$.
$$
\begin{CD}
T(p,q) @>Pinch>> T(r,s) \\ 
@VVV @VVV \\
\frac{p}{q} = [c_0,\dots, c_m] @>Step>> \frac{r}{s} =  [c_0,\dots, c_m-2]
\end{CD}
$$

\end{proposition}

\begin{proof}

We must separately handle the trivial case where $q=1$. But first, let us consider the case $q> 1$. Let $\frac{p}{q} = [c_0, \dots, c_m]$ be in canonical form, and observe that since $q>1$, it follows that $m\geq 1$. Let $p_i$ and $q_i$ be the numerator and denominator, respectively, of the $i^{th}$ convergent of $[c_0, \dots, c_m]$. 

A well known formula for convergents of continued fractions (see Theorem 2 in \cite{Khinchin:1997}) is:
$$p_{i}q_{i-1} - p_{i-1}q_{i} = (-1)^{i-1}$$
Applying this to $i=m$ and using the fact that $p = p_m$ and $q=q_m$, we obtain:
$$pq_{m-1} - p_{m-1}q = (-1)^{m-1}$$

Now we consider two cases, addressing the  parity of $m$. Suppose that $m$ is odd. Then we have $pq_{m-1} - p_{m-1}q = 1.$ Since $q=q_m=c_mq_{m-1}+q_{m-2}$, it follows that $1\le q_{m-1}\le q-1$. Similarly,  $1\le p_{m-1}\le p-1$. Altogether, this implies that $p_{m-1}$ and $q_{m-1}$ satisfy the requirements found in Lemma~\ref{LemmaOnTheProofOfBatsonsPinchingMoveFormula} for $t$ and $h$, respectively. Therefore, the values of $r$ and $s$ are given by $r=|p-2p_{m-1}|$ and $s=|q-2q_{m-1}|$. 

Working from the other direction, let us perform a step to the continued fraction of $\frac{p}{q} = [c_0,\dots, c_m] $. For the moment, let us denote the rational number corresponding to the resulting continued fraction $[c_0, \dots, c_m-2]$ by $\frac{a}{b}$. We will show that $\frac{a}{b} = \frac{r}{s}$, where $r$ and $s$ are the values we just computed above. 

Since the two continued fractions $[c_0, \dots, c_m]$ and $[c_0, \dots, c_m-2]$ agree in the first $m$ entries, it follows that the $i^{th}$ convergents coincide for $0\leq i\leq m-1$. The recursive relations \eqref{EquationRecursiveDefinitionOfTheithConvergent} then imply the following:
\begin{align*}
a = a_m & = (c_m-2)a_{m-1}+a_{m-2} =c_mp_{m-1}+p_{m-2} - 2p_{m-1} = p_m-2p_{m-1} = p-2p_{m-1},\cr
b = b_m & = (c_m-2)b_{m-1}+b_{m-2} =c_mq_{m-1}+q_{m-2} - 2q_{m-1} = q_m-2q_{m-1} = q-2q_{m-1}. 
\end{align*} 
Hence $\frac{a}{b} = [c_0, \dots, c_m-2] = \frac{r}{s}$, as desired. This  concludes the case where $m$ is odd.

Now if $m$ is even (and recall that $m \geq 1$), then  $pq_{m-1}-qp_{m-1} = -1$. Therefore, $p(q-q_{m-1}) - q(p-p_{m-1}) = 1$ with 
$1\le q-q_{m-1}\le q-1$ and $1\le p-p_{m-1}\le p-1$. This shows that $p-p_{m-1}$ and $q-q_{m-1}$ satisfy the requirements found in Lemma~\ref{LemmaOnTheProofOfBatsonsPinchingMoveFormula} for $t$ and $h$, respectively. So we conclude that
$r=|p-2p_{m-1}|$ and $s=|q-2q_{m-1}|$. Similar to the previous case, we can work from the other direction and find that the continued fraction $ [c_0, \dots, c_m-2]$  represents the rational number $\frac{p-2p_{m-1}}{q-2q_{m-1}}$, which completes the case.%

Finally, we consider the trivial case where $q = 1$. Performing a pinch move to the torus knot $T(p,1)$, one can check using Lemma~\ref{LemmaOnTheProofOfBatsonsPinchingMoveFormula} that the result is $T(p-2,1)$. On the other hand, the associated continued fraction is $\frac{p}{1} = [p]$. Performing a step on the continued fraction expansion, one obtains $[p-2] = \frac{p-2}{1}$, which also corresponds to the torus knot $T(p-2, 1)$, as desired. 
\end{proof}

No doubt this second characterization of pinch moves in terms of continued fractions may, at this moment, seem unnecessary. But in fact, this perspective will be valuable. We now consider the magnitude of the torus knot parameters after a pinch move. A variation of the following result is also found in Lemma 2.3 of ~\cite{JabukaVC:2018}.

\begin{proposition}\label{MagnitudeOfRandS}
Let $p,q>1$ be relatively prime integers such that if $pq$ is odd, we take $p>q$, and if $pq$ is even, we take $p$ to be even.
Let the torus knot $T(r,s)$ be obtained from the torus knot $T(p,q)$ by a pinch move. If $p>q$, then $r\geq s$. On the other hand, if $p<q$, then $r < s$. 
\end{proposition}

\begin{proof}
Suppose that $p>q$. It follows that the  canonical continued fraction expansion for $\frac{p}{q}$ is of the form $[c_0, c_1, \ldots, c_m]$ where $c_0$ is nonzero. It suffices to show that the canonical continued fraction expansion for $\frac{r}{s}$ also has a nonzero first entry. From Proposition~\ref{pinch-step}, we know that a continued fraction expansion for $\frac{r}{s}$ is $[c_0, c_1, \ldots, c_m-2]$, which perhaps is not yet in canonical form. 

Working through the possibilities, one can verify that when the expansion $\frac{r}{s} = [c_0, c_1, \ldots, c_m-2]$ is put into canonical form, the first entry is still nonzero. We describe these possibilities here. 

First, if $c_m\geq 4$, then the canonical form of the continued fraction is again \\$\frac{r}{s} = [c_0, c_1, \ldots, c_m-2]$, and it follows immediately that the first entry is nonzero. 

Next suppose that $c_m = 3$. Then the canonical form of the continued fraction is $\frac{r}{s} = [c_0, c_1, \ldots, c_{m-1}+1]$ if $m>0$, and it is $[1]$ if $m=0$. So again, the first entry is nonzero. 

We lastly consider the case $c_m=2$, and we consider the different possible values for $m$ in this case. First, if $m = 0$, then we have $\frac{p}{q} = [2] = \frac{2}{1}$ which means that $q=1$, which is precluded in our assumptions. It is also impossible that $m=1$, since that would imply $q=2$, but $q$ is odd. So then, it remains to consider the possibility $m\geq 2$.  We have $\frac{r}{s} = [c_0, c_1, \ldots, c_{m-1},0] = [c_0, c_1, \ldots, c_{m-2}]$. So the canonical form of the continued fraction expansion of $\frac{r}{s}$ is given by:

$$\frac{r}{s} ~ = ~\begin{cases}
[c_0, c_1, \ldots, c_{m-2}] & \text{ if $c_{m-2} > 1$} \\
[c_0, c_1, \ldots, c_{m-3}+1] & \text{  if $c_{m-2} = 1$ and $m\geq 3$}\\

[1] & \text{  if $c_{m-2} = 1$ and $m = 2$}\\
\end{cases}
$$
In any of these cases, since $c_0$ is nonzero, it follows that the first entry of the continued fraction expansion of $\frac{r}{s}$ in canonical form is nonzero, which implies that $r\geq  s$, as desired.

Now, suppose that $p<q$. From our convention, it follows that $p$ is even and $q$ is odd. The canonical continued fraction expansion for $\frac{p}{q}$ is of the form $[0, c_1, \ldots, c_m]$.  It suffices to show that the continued fraction expansion for $\frac{r}{s}$ also has a first coefficient of zero. Similar to before, one can work through the possibilities to show that once the expansion of $\frac{r}{s}=[0, c_1, \ldots, c_m-2]$ is put into canonical form, the first entry is still zero, which implies that $r<s$, as desired. We leave these details to the reader.

\end{proof}

Using the notation in Lemma~\ref{LemmaOnTheProofOfBatsonsPinchingMoveFormula},  the torus knot resulting from a single pinch move on $T(p,q)$ is the torus knot $T(|p-2t|, |q-2h|)$. It turns out that the signs of $p-2t$ and $q-2h$ coincide.

\begin{lemma} \label{LemmaAboutTheSameSignaOfRAndS}
Let $p,q>1$ be relatively prime integers. If $pq$ is odd, let $p>q$. If $pq$ is even, then let $p$ be even. Let $t, h$ be obtained from $p,q$ as in Lemma~\ref{LemmaOnTheProofOfBatsonsPinchingMoveFormula}. Then 
$$(p-2t)(q-2h)\ge0,$$ 
with equality occurring if and only if $T(p,q) = T(2, \ell)$ for some odd integer $\ell$.
\end{lemma}

The result above is proved in~\cite{JabukaVC:2018}. This leads us to the following definition.
\begin{definition}
Let $p,q$ be as in Lemma~\ref{LemmaOnTheProofOfBatsonsPinchingMoveFormula}. 
A pinch move on $T(p,q)$ is {\em positive} (respectively, {\em negative}) if  $p-2t$ and $q-2h$ are both positive or zero (respectively, both negative or zero). 

\end{definition}

We now give a characterization of when pinch moves are positive or negative in terms of the continued fraction expansion of $\frac{p}{q}$.

\begin{theorem}\label{PositiveNegativePinchMove}
Let $T(p,q)$ be a nontrivial torus knot with $q \neq 2$. A pinch move applied to $T(p,q)$ will be a positive pinch move if and only if the canonical continued fraction expansion is of the form $\frac{p}{q} = [c_0, \dots, c_m]$, where $m$ is odd. 
\end{theorem}

\begin{proof}

Suppose $m$ is odd. Then by the proof of Proposition~\ref{pinch-step}, the associated values of $t$ and $h$ involved in the pinch move are related to the convergents of $\frac{p}{q}$ as follows: $t =p_{m-1}$ and $h = q_{m-1}$. It suffices to show that $p-2t = p - 2p_{m-1}$ is {\em positive}.

In the special case that $m = 1$, the continued fraction is of the form $\frac{p}{q} = [c_0, c_1] = \frac{c_0c_1 + 1}{c_1}$, where $c_1\geq 2$. So we have:
$$p - 2t = p - 2c_0 = (c_0c_1 + 1) - 2c_0 = c_0(c_1 - 2) + 1 > 0.$$ 
Hence the pinch move is positive.
 
If $m > 1$ (and recall that $m$ is odd), then using the fact that $c_m\geq 2$ and using Equation~\ref{EquationRecursiveDefinitionOfTheithConvergent}, we have:
$$p - 2t = p - 2p_{m-1} \geq p - c_mp_{m-1} = p_m - c_mp_{m-1} = p_{m-2} > 0.$$
Therefore the pinch move is positive. Thus we have shown that if $m$ is odd, then the pinch move applied to $T(p,q)$ is positive.

Now suppose $m$ is even. To show that the pinch move on $T(p,q)$ is negative, it suffices to show that $q-2h$ is negative. Notice that if $m = 0$, then the associated torus knot is an unknot, which we have precluded. Hence $m > 0$ (and recall that $m$ is even). From the proof of Proposition~\ref{pinch-step}, we know that in this case, $t =p -p_{m-1}$ and $h = q - q_{m-1}$.  We must show that $q-2h = q - 2(q-q_{m-1})$ is {\em negative}. Again using the fact that $c_m\geq 2$ and Equation~\ref{EquationRecursiveDefinitionOfTheithConvergent}, we have 
$$q - 2h = q - 2(q -q_{m-1}) = 2q_{m-1} - q \leq c_mq_{m-1} - q = c_mq_{m-1} - q_m  = -q_{m-2} < 0$$
Therefore the pinch move is negative.  

\end{proof}

\section{A surface construction in $B^4$}~\label{surface-construction}

In this section, we make use of the band move defined in the previous section to construct a nonorientable surface with boundary $T(p,q)$.

In general, beginning with any knot $K$, if a band move on $K$ results in the knot $K'$, there is a corresponding smooth cobordism from $K$ to $K'$. In our particular setting, the pinch move on $T(p,q)$ gives a nonorientable cobordism between two torus knots $T(p,q)$ and the resulting knot $T(r,s)$. The cobordism can be realized in $T^2\times [0, 1]$, with $T(p,q)$ lying on the torus $T^2\times\{0\}$ and with  $T(r,s)$ lying on the torus $T^2\times\{1\}$. Moreover, the parity and relative magnitude of the integers in the pair $(p,q)$ is the same as that of $(r,s)$ (Proposition~\ref{MagnitudeOfRandS}), so the ordering of the pair $(r,s)$ coincides with our convention (namely, if $rs$ is even, then $r$ is even, and if $rs$ is odd, then $r>s$). 

Repeatedly applying pinch moves, we obtain a sequence of torus knots:
$$T(p,q)\xrightarrow{\text{ Pinch}} T(r_1, s_1) \xrightarrow{\text{ Pinch}} T(r_2, s_2) \xrightarrow{\text{ Pinch}} \cdots \xrightarrow{\text{ Pinch}} T(r_n, s_n)$$
And this sequence of torus knots, in turn, represents a smooth cobordism from $T(p,q)$ to $T(r_n, s_n)$. Furthermore, we can take the cobordism to live in $T^2 \times  [0, 1] \subset S^3$. 

Geometrically, one can observe that if one starts with a nontrivial torus knot $T(p,q)$, then the parameters of the torus knot $T(r,s)$ resulting from a single pinch move are smaller than that of the original torus knot. Hence, we can be sure that eventually the sequence of pinch moves applied to $T(p,q)$ will produce the unknot. In this way, we have produced a nonorientable cobordism from $T(p,q)$ to an unknot $T(\ell, 1)$ for some $\ell\geq 0$. (We will pin down the value of $\ell$ in Theorem~\ref{p0}.) The cobordism from $T(p,q)$ to $T(\ell,1)$ can be capped off with a disk which is embedded in $B^4$. The resulting nonorientable surface is denoted $F_{p,q}$ and has boundary $T(p,q)$, as desired. 

The first Betti number of this surface $F_{p,q}$ equals the number of pinch moves applied to first reduce $T(p,q)$ to an unknot $T(\ell,1)$. Equivalently, the first Betti number of $F_{p,q}$ equals the number of steps needed to reduce the continued fraction expansion for $\frac{p}{q}$ to an integer $\ell$. Let us consider an example.

\begin{example}
The torus knot $T(4,3)$ reduces to the unknot $T(2,1)$ with just one pinch move. Therefore there is a nonorientable cobordism between $T(4,3)$ and $T(2,1)$. Gluing a disk to the cobordism along the unknot $T(2,1)$, we obtain a nonorientable surface $F_{4,3}$. Notice that $\beta_1(F_{4,3}) = 1$, which implies that $F_{4,3}$ is a M\"obius band embedded in $B^4$. Since $\gamma_4$ is bounded below by 1, it follows that $\gamma_4(T(4,3)) =1$. See Fig.~\ref{SurfaceFigure-F-3-4} for an illustration of the surface $F_{4,3}$ immersed in $S^3$. 
\end{example}

Next, we give a formula for the value of $\ell$ in the construction of $F_{p,q}$.

\begin{theorem}\label{p0}
Let $p,q >1$ be relatively prime positive integers such that if $pq$ is odd, then $p>q$, and if $pq$ is even, then $p$ is even. Write $p = qk + a $, where $0< a < q$ and $k\geq0$. 

Apply pinch moves to the torus knot $T(p,q)$ until it first becomes unknotted. The resulting unknot arising from the sequence of pinch moves is $T(\ell, 1)$, where $\ell$ is given by the formula below: 

\begin{equation}\label{ell}
\ell = \begin{cases}

k & \text{ if $p\equiv k\!\!\!\!\pmod{2}$} \\
k + 1 & \text{ if $p\not\equiv k\!\!\!\! \pmod{2}$}
\end{cases}
\end{equation}

\end{theorem}
\begin{proof}
Since $p = qk+a$, it follows that the first entry in the continued fraction expansion of $\frac{p}{q}$ is $k$. We represent the canonical continued fraction expansion as follows: $\frac{p}{q} = [k, c_1, c_2, \ldots, c_m]$. By Proposition~\ref{pinch-step}, the value of $\ell$ is determined by applying steps to the continued fraction until the continued fraction is first reduced to an expansion with a single entry. However, considering the effect that a step has on the continued fraction (Definition~\ref{StepDefinition}) and the possible changes that occur when putting the resulting expansion in canonical form (Equation~\ref{EquationMakingTheContinuedFractionExpansionCanonical}), there are only two possibilities for what this terminal continued fraction expansion can be. It must be either $[k]$ or $[k+1]$. These two possible continued fractions correspond to the rational numbers $\frac{k}{1}$ or $\frac{k+1}{1}$, respectively. These rational numbers, in turn, correspond to the torus knots $T(k, 1)$ and $T(k+1,1)$, respectively. Since we know that the parities of the torus knot parameters are preserved by pinch moves, the resulting torus knot is uniquely determined by the parity of $p$. In particular, if  $p\equiv k\pmod{2}$, then the associated torus knot is $T(k,1)$, and hence $\ell = k$. On the other hand, if  $p\not\equiv k \pmod{2}$, then the associated torus knot is $T(k+1,1)$, and hence $\ell = k+1$. 
\end{proof}

In general, the surface $F_{p,q}$ does not realize $\gamma_4(T(p,q))$~\cite{JabukaVC:2018, Lobb:2019}. However, there exists an infinite family of torus knots for which $F_{p,q}$ does realize $\gamma_4$. Such families are given in \cite{Batson:2014}, as well as the following result.

\begin{theorem}~\cite{JabukaVC:2018}\label{AllPositivePinchMoves}
Let $p,q >1$ be relatively prime positive integers such that $p$ is even. If every pinch move in the construction of $F_{p,q}$ is a positive pinch move, then $\gamma_4(T(p,q)) = \beta_1(F_{p,q})$. 
\end{theorem}

\section{Two surface constructions in $S^3$ and the value of $\gamma_3(T(p,q))$}~\label{3D-constructions}

In this section, we  tweak the surface construction from Section~\ref{surface-construction} in order to obtain a nonorientable surface embedded in $S^3$ with boundary $T(p,q)$. We must consider the cases when $pq$ is even and odd separately.  The two constructions we describe here realize the nonorientable three-genus of torus knots, as computed by Teregaito~\cite{Teragaito:2004}. 
 
\begin{figure}   
\centerline{\includegraphics[width=4.5cm]{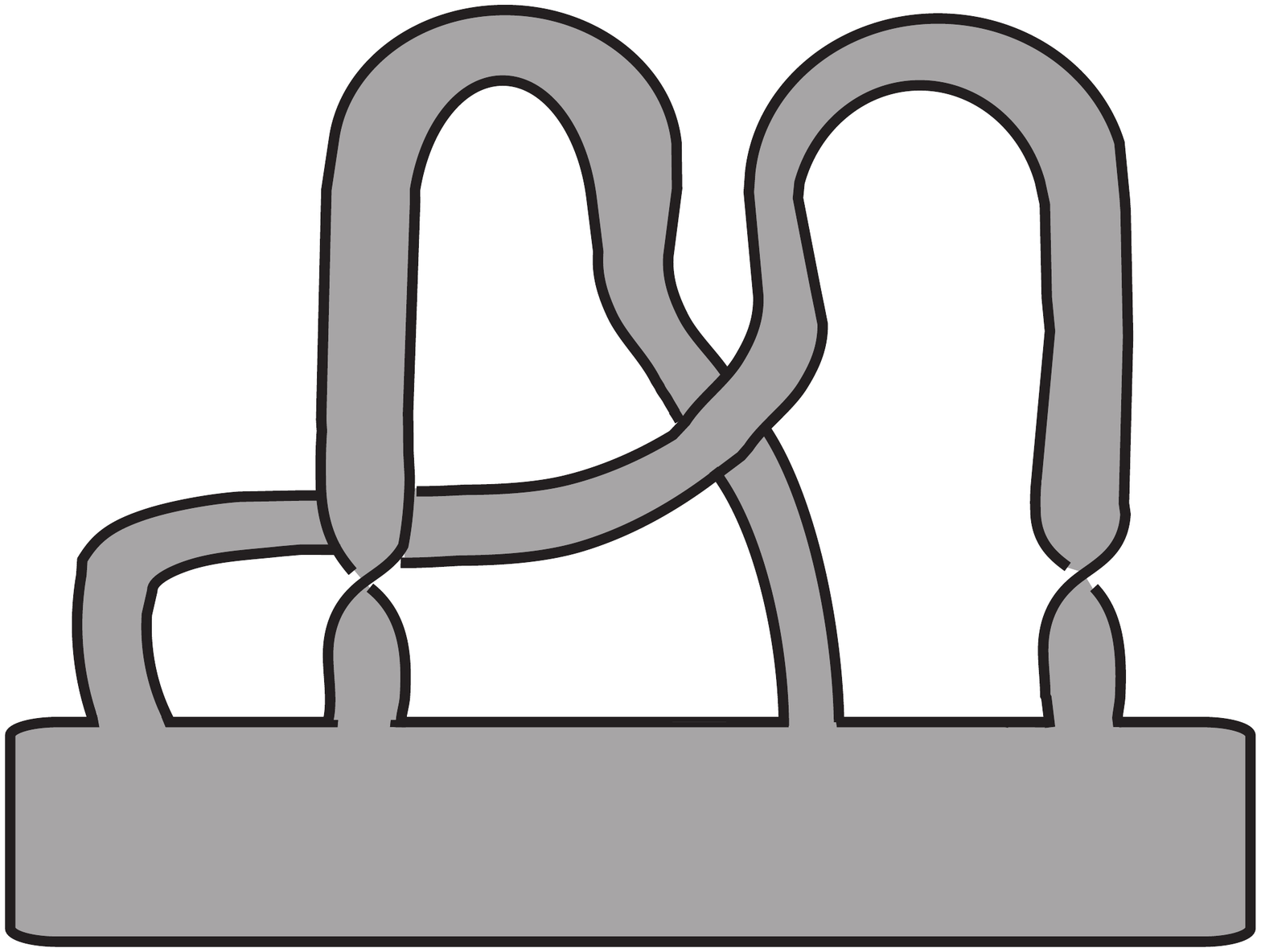} \qquad 
\includegraphics[width=3cm]{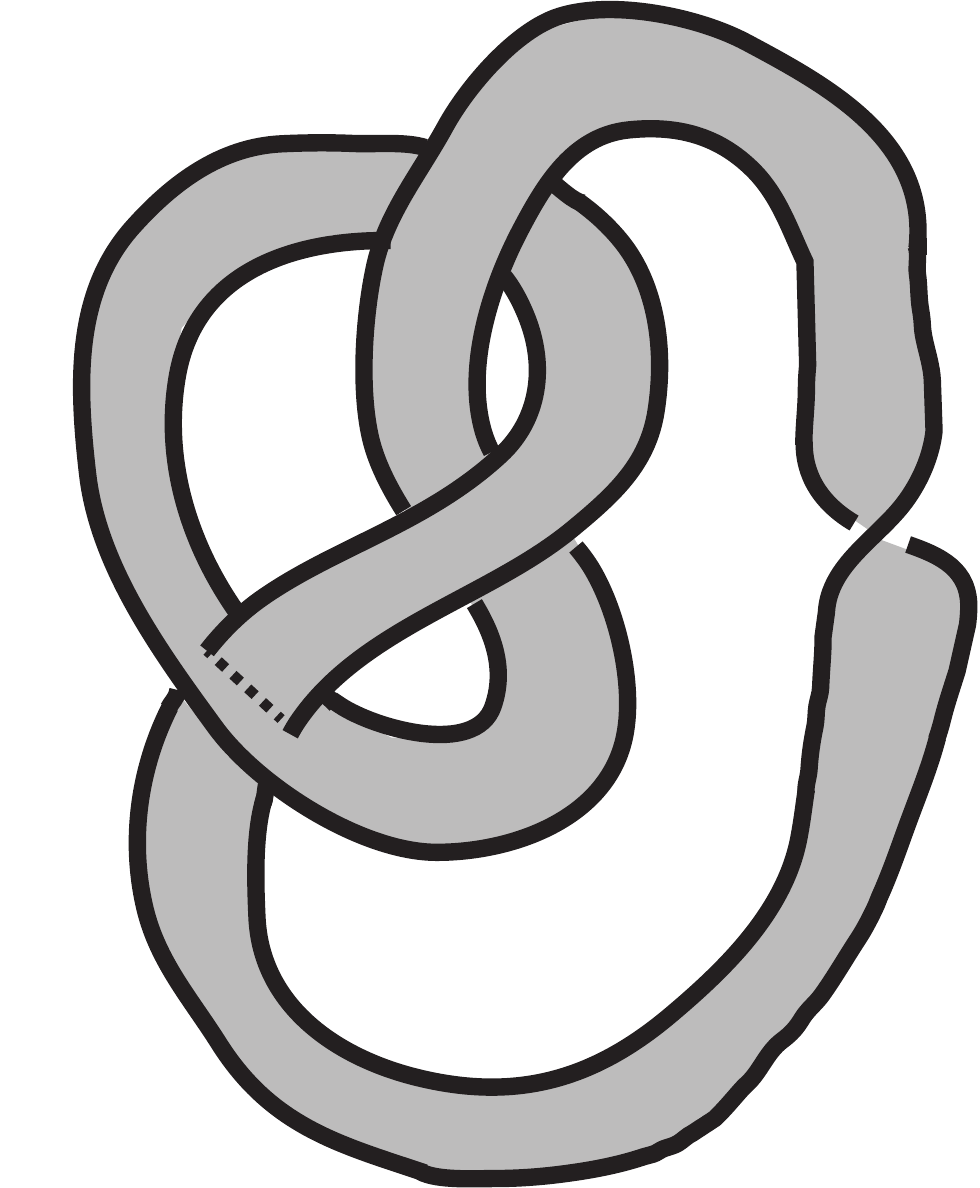}}
\vspace*{8pt}
\caption{These two surfaces both have boundary equal to $T(4,3)$. On the left, we have the nonorientable surface $\Sigma_{4,3}$ which realizes $\gamma_3(T(4,3)) = 2$. On the right, we have the nonorientable surface $F_{4,3}$ (immersed in $S^3$) which realizes $\gamma_4(T(4,3))= 1$.  }
\label{SurfaceFigure-F-3-4}
\end{figure}


\subsection{A surface $\Sigma_{p,q}$ in $S^3$, when $pq$ is even}~\label{Sigma-p-q}

We first consider the case that $pq$ is even, and we set $p$ to be even and $q$ odd. Consider the torus knot $T(p,q)$ to be embedded on the torus $T^2 \times \{0\} \subset T^2\times[0,1]$ in such a way that the torus knot wraps $p$ time about the longitude of the torus and $q$ times about the meridian of the torus. 

Recall that the parity of the torus knot parameters are preserved with pinch moves (see Lemma~\ref{LemmaOnTheProofOfBatsonsPinchingMoveFormula}). Therefore after applying pinch moves to $T(p,q)$ and obtaining an unknot $T(\ell, 1)$, it must be the case that $\ell$ is even. One can check that a pinch move applied to $T(\ell, 1)$ will result in the torus knot $T(\ell-2,1)$. Continuing in this way, we will eventually reach the unknot $T(0,1)$. In this case, observe that $T(0,1)$  bounds a disk in $S^3$ that embeds into the complement of the cobordism, since the cobordism itself is contained in $T^2 \times [0,1]$, and the torus knot $T(0,1)$ wraps once about the meridian of $T^2 \times \{1\} \subset T^2 \times [0,1]$.  So the cobordism can be capped off with a disk in $S^3$. Thus, we have a nonorientable surface in $S^3$ with boundary $T(p,q)$ and with first Betti number equal to the number of pinch moves performed in reducing $T(p,q)$ to $T(0,1)$. We denote the surface by $\Sigma_{p,q}$. Moreover, by Proposition~\ref{pinch-step}, we have:
$$\beta_1(\Sigma_{p,q}) = \text{\small number of steps needed to reduce the continued fraction expansion of $\frac{p}{q}$ to 0}.$$

\begin{example}
Consider the torus knot $T(4,3)$. Applying pinch moves to the knot, we obtain the following sequence:
$$T(4,3) \xrightarrow{\text{ Pinch}} T(2,1) \xrightarrow{\text{ Pinch}} T(0,1)$$
Therefore $T(4,3)$ bounds a nonorientable surface $\Sigma_{4,3} \subset S^3$ with $\beta_1(\Sigma_{4,3}) = 2$. The surface is illustrated on the left of Fig.~\ref{SurfaceFigure-F-3-4}.\\

\end{example}


\subsection{A surface $G_{p,q}$ in $S^3$, when $pq$ is odd}~\label{G-p-q}

Finally, we consider the case that $pq$ is odd. We take $p>q$. Unfortunately, the construction in $S^3$ we just described for the case $pq$ even does not work here. The pinch moves applied to $T(p,q)$ will terminate in an unknot $T(\ell,1)$, as before, but $\ell$ will be odd, and further pinch moves will reduce this unknot to $T(1,1)$. The associated cobordism from $T(p,q)$ to $T(1,1)$ in $T^2\times [0, 1]\subset S^3$ cannot be capped off with a disk in $S^3$ without causing the surface to be immersed. Hence, we must proceed differently.

We follow the construction outlined in~\cite{Teragaito:2004} (see page 229). Consider the  torus knot $T(p,q)$ as lying on the surface of a solid torus $V\subset S^3$ with complementary solid torus $W$.   Let $\gamma$ be an arc as in Fig.~\ref{Gamma}, connecting two adjacent strands of the torus knot on the surface of the solid torus. Observe that $\partial \gamma$ splits the torus knot into two arcs, call them $A_1$ and $A_2$. Gluing $\gamma$ to either of these arcs creates a new torus knot: $T_1 = A_1 \cup \gamma$ and $T_2 = A_2 \cup \gamma$.  See Fig.~\ref{Gamma}. 

The values of the torus knot parameters for $T_1$ and $T_2$ are determined easily from the continued fraction expansion of $\frac{p}{q}$, as described  in~\cite{Teragaito:2004}. We review the result here.

\begin{proposition}\cite{Teragaito:2004}
Let $T(p,q)$ be a torus knot with $pq$ odd, and let $\gamma$ be an arc as in Fig.~\ref{Gamma}, connecting two adjacent strands of the torus knot such that the boundary of $\gamma$ splits the knot into two arcs: $A_1$ and $A_2$. The two resulting torus knots $T_1 = A_1 \cup \gamma$ and $T_2 = A_2 \cup \gamma$ have the following parameters (listed as unordered pairs):
$$T_1 = T(p_{m-1}, q_{m-1}) \hspace{.3cm} \text{ and }   \hspace{.3cm}T_2 = T(p-p_{m-1}, q - q_{m-1})$$
where $\frac{p_{m-1}}{q_{m-1}}$ is the $(m-1)^{st}$ convergent of $\frac{p}{q} = [c_0, c_1, \ldots c_m]$. 
\end{proposition}

 Observe that $p_{m-1}$ and $q_{m-1}$ must have opposite parity since $p$ and $q$ are both odd and satisfy $pq_{m-1} - p_{m-1}q = (-1)^{m-1}.$ Hence both of the torus knots $T_1$ and $T_2$ have an even parameter. Let $a$ denote the even integer from the set $\{p_{m-1}, q_{m-1}\}$, and let $b$ denote the odd integer from that set.  By our previous work in Section~\ref{Sigma-p-q}, the torus knot $T_1$ bounds a nonorientable surface $\Sigma_{a,b}$ in the standard solid torus $V$. Let $c$ denote the even integer in the set $\{p-p_{m-1}, q-q_{m-1}\}$, and let $d$ denote the odd integer in the set. Then $T_2$ bounds a nonorientable surface $\Sigma_{c,d}$ in the complementary solid torus $W\subset S^3$. These two surfaces intersect at $\gamma$. Now glue the two surfaces together along $\gamma$. Let us call the result $G_{p,q} = \Sigma_{a,b} \cup_\gamma \Sigma_{c,d}$. Observe that $\partial G_{p,q} = T(p,q)$. It follows that: 
$$\beta_1(G_{p,q}) = \beta_1(\Sigma_{a,b}) + \beta_1(\Sigma_{c,d}),$$
where  $\{a,b\} = \{p_{m-1}, q_{m-1}\}$ and $\{c,d\} = \{p - p_{m-1}, q - q_{m-1}\}.$

\begin{example}
Consider the torus knot $T(5,3)$. Adding in the arc $\gamma$ splits the torus knot into the two torus knots $T(2,1)$ and $T(2,3)$. See Fig.~\ref{Gamma}. The surface $G_{5,3}$, then, is given by $G_{5,3} = \Sigma_{2,1} \cup_\gamma \Sigma_{2,3}$. See Fig.~\ref{G-5-3}. Both $T(2,1)$ and $T(2,3)$  reduce to $T(0,1)$ with just one pinch move. So we have 
$$\beta_1(G_{5,3}) = \beta_1({\Sigma_{2,1}}) + \beta_1({\Sigma_{2,3}}) =  1+1 = 2.$$ 

\end{example}

\begin{figure}  
\centerline{\includegraphics[width=8cm]{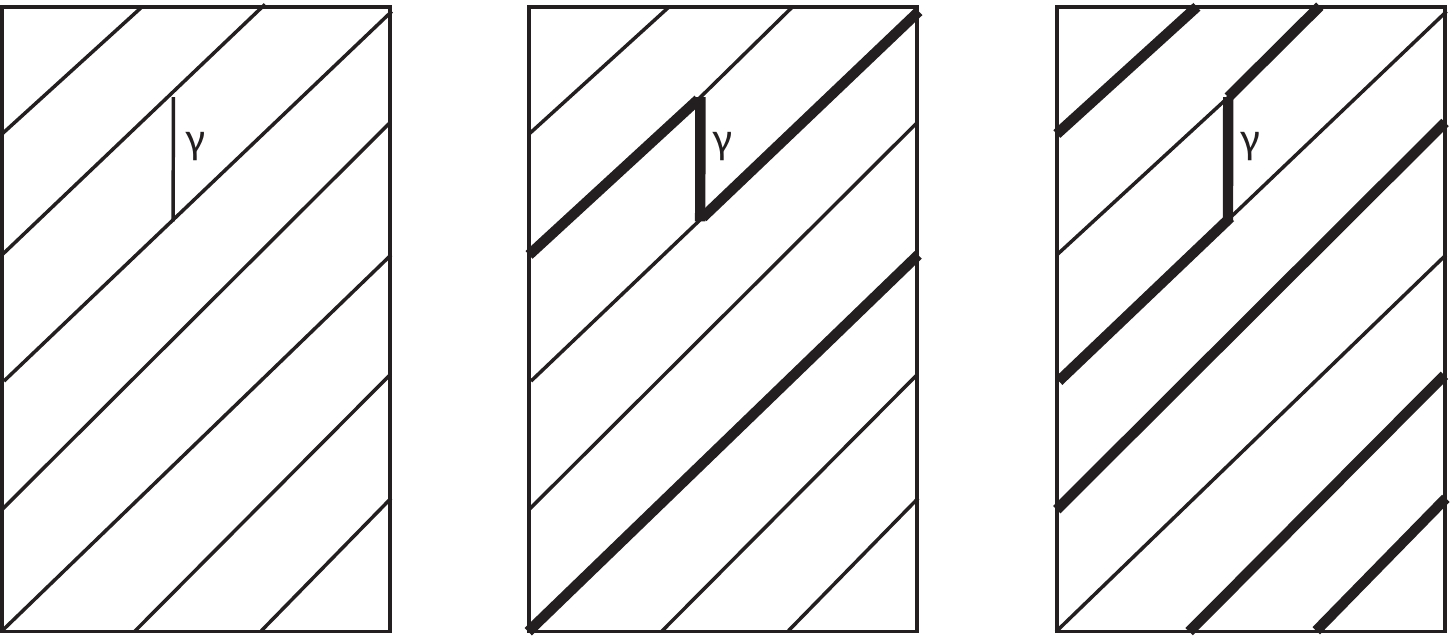}}
\vspace*{8pt}
\caption{Beginning with the torus knot $T(5,3)$ at left, we add an arc $\gamma$ connecting two adjacent strands. The endpoints of $\gamma$ split $T(5,3)$ into two arcs. Gluing $\gamma$ to each of these arcs, in turn, creates two  torus knots: $T(2,1)$ (at center) and $T(3,2)$ (at right).   }
\label{Gamma}
\end{figure}

\begin{figure}   
\centerline{\includegraphics[width=4cm]{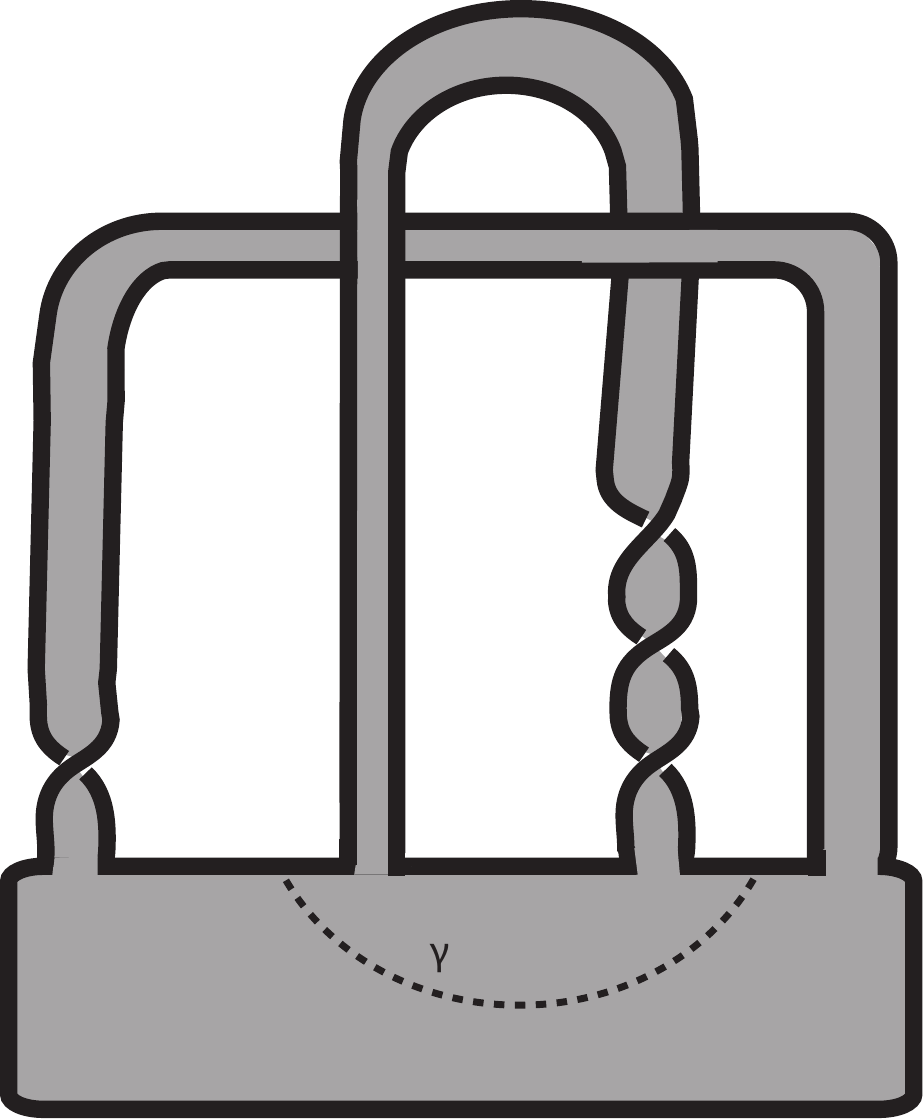}}
\vspace*{8pt}
\caption{The nonorientable surface $G_{5,3}$ with boundary $T(5,3)$ realizes $\gamma_3(T(5,3)) = 2$.  Observe that if the surface is cut along the curve $\gamma$, then the surface breaks into two surfaces, one with boundary $T(2,3)$ and the other with boundary $T(2,1)$.  }\label{G-5-3}
\end{figure}


\subsection{The nonorientable three-genus of torus knots}
The two surface constructions described above --  $\Sigma_{p,q}$ (if $pq$ is even) and $G_{p,q}$ (if $pq$ is odd) -- realize the value of $\gamma_3(T(p,q))$. This follows from work of Teragaito~\cite{Teragaito:2004}. Teragaito computed $\gamma_3$ for torus knots $T(p,q)$ as a function of steps on continued fractions. Before stating his results, we make a definition. The following function was originally defined by Bredon--Wood in (see Section 10 of~\cite{BredonWood:1969})  and has received other mutually equivalent descriptions (see Theorem 6.1 of~\cite{BredonWood:1969} and also page 225 of~\cite{Teragaito:2004}), but the definition below is particularly useful in our setting.  

\begin{definition}
Suppose that $a,b>0$ are relatively prime integers with $a$ even. We define $N(a,b)$ to be the number of successive steps needed to reduce the continued fraction expansion of $\frac{a}{b}$ to 0. 
\end{definition}

With this preliminary in place, we can now state Teragaito's results on nonorientable three genus of torus knots. 

\begin{theorem}\label{gamma3Values}~\cite{Teragaito:2004} Let $p,q>1$ be relatively prime. If $pq$ is even, let $p$ be even. If $pq$ is odd, let $p>q$.  Let $x$ be the unique solution of $xq\equiv -1\,(\text{mod }p)$ that lies in $\{1,\dots, p-1\}$. Then 
$$\gamma_3(T(p,q)) = 
\left\{
\begin{array}{cl}
N(p,q) & \quad ; \quad pq \text{ is even,}\cr  
N(pq-1,p^2) & \quad ; \quad pq \text{ is odd and } x \text{ is even,}\cr 
N(pq+1,p^2) & \quad ; \quad pq \text{ is odd and } x \text{ is odd.}\cr 
\end{array}
\right.
$$
\end{theorem}

One can check that the surfaces $\Sigma_{p,q}$ (in the case that $pq$ is even) and $G_{p,q}$ (in the case that $pq$ is odd) have first Betti number equal to these values given above  (see ~\cite{Teragaito:2004} for details). Hence, the surfaces realize $\gamma_3(T(p,q))$.
 

\section{Comparing $\gamma_3$ and $\gamma_4$}~\label{compare3and4}
In this section, we use the surface constructions from the previous sections to prove Theorem~\ref{theorem}. In fact, Theorem~\ref{theorem} follows from the following result.

\begin{theorem}\label{comparing} Let $p, q > 1$ be relatively prime integers with $p$ even, and let $F_{p,q}$ be the nonorientable surface from Section~\ref{surface-construction}. Write $p = qk + a $, where $0< a < q$ and $k\geq0$. Then
$$\gamma_3(T(p,q))-\beta_1(F_{p,q})=
 \begin{cases}
\frac{k}{2} & \text{ if $k$ is even}, \\
\frac{k + 1}{2} & \text{ if $k$ is odd}.
\end{cases}
$$
\end{theorem}
\begin{proof}

The constructions of surfaces $\Sigma_{p,q}$ (which realizes $\gamma_3$) and $F_{p,q}$ coincide at first. In both constructions, one performs pinch moves to the torus knot $T(p,q)$ until it first becomes unknotted. As discussed in Section~\ref{surface-construction}, the associated cobordism will be from $T(p,q)$ to $T(\ell, 1)$, where $\ell$ is an even integer given by the equation in Theorem~\ref{p0}. At this point, the construction of $F_{p,q}$ is finished by simply capping off with a disk embedded in $B^4$. Moreover, if $\ell = 0$, then $\Sigma_{p,q}$ is also concluded at this point by gluing a disk embedded in $S^3$ to $T(0,1)$. However, if $\ell$ is nonzero, then one must continue with the construction of $\Sigma_{p,q}$ by doing $\frac{\ell}{2}$ further pinch moves. Hence the construction of $\Sigma_{p,q}$ (which realizes $\gamma_3$) contains $\frac{\ell}{2}$ additional pinch moves compared to $F_{p,q}$, where the value of  $\ell$ is given by Theorem~\ref{p0}. The result follows. 
\end{proof}

Since $\gamma_4(T(p,q)) \leq \beta_1(F_{p,q})$, Theorem~\ref{theorem} immediately follows as a corollary of the above result.  To conclude, we give two families of torus knots where the difference $\gamma_3 - \gamma_4$ can be explicitly computed.

\begin{example}
The special case where $q=3$ is straightforward since the surface $F_{p,3}$ is a M\"obius band for all values of $p$. Hence $\gamma_4(T(p,3)) = \beta_1(F_{p,3}) = 1$. Write $p = 3k + a$ where $a = 1$ or $2$. If $p$ is even, the value of $\gamma_3(T(3,p))$ is easily computed using Theorem~\ref{gamma3Values}, and the difference $\gamma_3(T(p,3)) - \gamma_4(T(p,3))$ precisely the value stated in Theorem~\ref{theorem}. That is, 
$$\gamma_3(T(p,3)) - \gamma_4(T(p,3)) =   \begin{cases}
\frac{k}{2} & \text{ if $k$ is even}, \\
\frac{k + 1}{2} & \text{ if $k$ is odd}.
\end{cases}
$$


\end{example}

\begin{example}
Let $m,k$ be odd integers with $m>1$ and $k\geq 1$. Consider the family of torus knots
$ T(p,q) = T(km+1,m).$

The rational number $\frac{km+1}{m}$ has continued fraction expansion: $\frac{km+1}{m} = [k,m]$. Applying steps to this continued fraction expansion, we observe:

$$\frac{km+1}{m}=[k,m] \xrightarrow{\text{ Step}} [k,m-2] \xrightarrow{\text{Step}} [k,m-4] \xrightarrow{\text{Step}} \cdots \xrightarrow{\text{Step}}[k,1] = [k+1] = k+1   $$ 

Hence, a total of $\frac{m-1}{2}$ steps will reduce the fraction $\frac{km+1}{m}$ down to the integer $k+1$. Therefore, by Proposition~\ref{pinch-step}, the surface $F_{p,q}$ is constructed with $\frac{m-1}{2}$ pinch moves, and $\beta_1(F_{p,q}) = \frac{m-1}{2}$. Furthermore, observe that every pinch move in the construction of this surface was {\em positive} (Theorem~\ref{PositiveNegativePinchMove}), and therefore by Theorem~\ref{AllPositivePinchMoves}, it follows that $\gamma_4(T(p,q)) = \beta_1(F_{p,q}) = \frac{m-1}{2}$.  

An additional $\frac{k+1}{2}$ steps will reduce $[k+1]$ down to $[0]$. Therefore by Theorem~\ref{gamma3Values}, $\gamma_3(T(p,q))  =  \frac{m-1}{2} + \frac{k+1}{2}$. It follows that $(\gamma_3 - \gamma_4) (T(p,q)) = \frac{k+1}{2}$. 

\end{example}

\section*{Acknowledgments}
We wish to thank Josh Batson for helpful email correspondence with the second author. The second author also thanks Robert Lipshitz for asking good questions that led to further insights. The first author was partially supported by a grant from the Simons Foundation, Award ID 524394, and a grant from the National Science Foundation, DMS-1906413.

\end{document}